\documentclass[a4paper,leqno,11pt]{article}

\usepackage{amsmath,amssymb,amsthm,%showkeys,
verbatim,%refcheck
}

\usepackage[a4paper,top=3cm,bottom=3cm,left=3.2cm,right=3cm]{geometry}
\usepackage{tikz}
\tikzstyle{nodo}=[circle,draw,fill,inner sep=0pt, minimum size=0.5*width("k")]
\tikzstyle{infinito}=[circle,inner sep=0pt,minimum size=0mm]
\newcommand\infvert{$\scriptstyle\infty$}

\newcommand\R{{\mathbb R}}

\newcommand\weak{\rightharpoonup}

\newcommand\G{\mathcal G}

\newcommand\infV{\Lambda}

\newcommand\HH{\mathcal H}
\newcommand\be{\begin{equation}}
\newcommand\ee{\end{equation}}
\newcommand\Hmu{{H^1_\mu}}
\newcommand\Vmu{{V_\mu}}

\newcommand\vv{\textsc{v}}

\newcommand\eps{\varepsilon}

\newtheorem{theorem}{Theorem}[section]

\newtheorem{lemma}[theorem]{Lemma}
\newtheorem{corollary}[theorem]{Corollary}

\theoremstyle{remark}
\newtheorem{remark}[theorem]{Remark}
\newtheorem*{remark*}{Remark}
\newtheorem{example}[theorem]{Example}

\theoremstyle{definition}

\date{}

\title{Multiple positive bound states \\ for the subcritical NLS equation on metric graphs }

\author{Riccardo Adami, Enrico Serra, Paolo Tilli \\ \ \\{\small  Dipartimento di Scienze
Matematiche ``G.L. Lagrange'', Politecnico di Torino } \\ {\small
Corso Duca degli Abruzzi, 24, 10129 Torino, Italy}}

\begin{document}

\maketitle

\begin{abstract}
We consider the Schr\"odinger equation with a subcritical focusing power
nonlinearity on a noncompact metric graph, and prove that for every finite edge
there exists a threshold value of the mass, beyond which there exists a positive
bound state achieving its maximum on that edge only. This bound
state is characterized as a minimizer of the energy functional
associated to the NLS equation, with an additional constraint (besides
the mass prescription): this requires particular care in proving that
the minimizer satisfies the Euler?Lagrange equation. As a
consequence, for a sufficiently large mass every finite edge of the
graph hosts at least one positive bound state that, owing to its
minimality property, is orbitally stable.
 \end{abstract}

\smallskip


\section{Introduction}
Given $p\in (2,6)$, the existence of ground states for the NLS energy functional
\begin{equation}
\label{NLSe}
E (u,\G)
  =  \frac 1 2 \| u' \|^2_{L^2 (\G)}
- \frac 1 p  \| u \|^p_{L^p (\G)}
=\frac 1 2 \int_\G |u'|^2dx
-\frac 1 p \int_\G |u|^pdx
\end{equation}
on a noncompact
metric graph $\G$, under the
{\em mass constraint}
\begin{equation}
\label{mass}
\| u \|^2_{L^2 (\G)} \ = \ \mu,
\end{equation}
has been been investigated in a series of papers \cite{acfn, acfn2, ast, ast2, ast3}.
Several conditions on $\G$ and $\mu$ have been found,
which may guarantee (or, on the contrary,
rule out) the existence of absolute minimizers of $E(u,\G)$
in the mass--constrained space
\begin{equation}
\label{Hmu}
H^1_\mu(\G):=\left\{ u\in H^1(\G) :\,\Vert
u\Vert_{L^2(\G)}^2=\mu\right\}.
\end{equation}

In this paper, instead, we switch our focus from ground states to \emph{bound states},
that is, to functions $u\in H^1_\mu(\G)$ which are constrained \emph{critical points}
of the NLS energy functional, possibly without being absolute minimizers. More precisely,
given a mass $\mu>0$,
a bound state is a function $u\in \Hmu(\G)$ that satisfies the Euler--Lagrange equation
\begin{equation}
  \label{euler}
  \int_\G \left(-u'\eta'+u |u|^{p-2}\eta\right)\,dx=\lambda
\int_\G u\eta\,dx\quad
\forall\eta\in H^1(\G),
\end{equation}
where $\lambda$ is a Lagrange multiplier due to the mass constraint \eqref{mass}.
Of course, this means that $u$ solves
the nonlinear equation
\begin{equation}
\label{euleroforte} u''+u |u|^{p-2}=\lambda u
\end{equation}
on every edge  of $\G$, coupled with the homogeneous \emph{Kirchhoff
condition}
\begin{equation}
\label{kir}
\sum_{e\succ \vv}
\frac{d u}{d x_e}(\vv)=0
\end{equation}
at every vertex $\vv$ of $\G$ (the sum is extended to all edges $e$ incident at $\vv$,
see \cite{ast} for more details).
As usual, the existence of bound states is equivalent to the existence of stationary solutions of the corresponding NLS equation on
$\G$, and this becomes all the more relevant if
$\G$ has no ground state.

\smallskip

Before describing in detail our main results, we recall that
a \emph{noncompact metric graph} $\G$ is
a connected metric space obtained by gluing together,
by the identification of some of their endpoints (the ``vertices'' of $\G$), a finite number of closed
line intervals (not necessarily bounded), according to the topology of
a graph, self-loops and multiple edges being allowed.
Each edge $e$, after choosing a coordinate $x_e$ on it,
can be regarded either as an interval $[0,\ell_e]$, or as a positive half-line $[0,+\infty)$ (in this case
the edge is attached to $\G$ at $x_e=0$), and the spaces $L^r(\G)$, $H^1(\G)$ etc. can be
defined in a natural way (in particular, $u\in H^1(\G)$ means that $u\in H^1(e)$
for every edge $e$ of $\G$, with the additional requirement that
$u$ is \emph{continuous} on $\G$, i.e. it has no jump at any vertex:
we refer to \cite{ast,kuchment} for more details). Endowing $\G$ with the shortest path distance,
one obtains a locally compact metric space, and
$\G$ being noncompact
is equivalent to one or more edges being unbounded:
two very special cases are $\G=\R^+$ (i.e. $\G$ consists of just one unbounded edge)
and $\G=\R$  (obtained by gluing together two
unbounded edges).

It is within this framework that we look for bound states, i.e. solutions to \eqref{euler}
belonging to the space $H^1_\mu(\G)$ defined in \eqref{Hmu}.
In this investigation, we fix an exponent $p\in (2,6)$ (the \emph{critical case}
where $p=6$ is rather special and will not be considered here,
see \cite{ast3} for results on ground states),
while we will focus our attention on the mass $\mu$ as a varying parameter of the problem.

We will prove (Theorem \ref{mainteo} and Corollary \ref{coroll}) that if $\G$ is any noncompact metric graph
and $\mu$ is large enough,
then there exist at least as many bound states of mass $\mu$ as the number of bounded edges of $\G$. This is in contrast with the case of ground states,
which on some graphs may fail to exist for every value of $\mu$ (see \cite{ast,ast2}).

Our strategy  is a purely variational one, in that bound states are obtained as solutions
of a doubly-constrained minimization procedure for
the NLS energy \eqref{NLSe}.
More precisely,
we fix a bounded edge $e$ of $\G$
(if every edge of $\G$ is a halfline, then $\G$ is a star--graph and in this case the existence of
a bound state is already known, see \cite{acfn2}), and
we minimize $E(u,\G)$
among all functions $u\geq 0$ of mass $\mu$ which achieve
their absolute maximum on the edge $e$ (see \eqref{defV}, \eqref{defVmu}). It turns out
that a minimizer (i.e. a function $u$ satisfying \eqref{minpV}) exists,
provided $\mu$ is large enough (Theorem~\ref{minesiste}). Observe that
the existence of a minimizer is nontrivial: indeed, even though
the condition of achieving the maximum on $e$ is weakly closed (see Remark~\ref{remVcl}),
the class $V_\mu$ where we seek a minimizer
is far from
being compact, due to the mass constraint \eqref{mass}. Thus, to prove that a minimizer
exists, we first prove a general result (Theorem~\ref{propgen}), that
gives quantitative lower bounds
for the NLS energy along a  sequence of functions
$u_n\weak u$ in $H^1_\mu(\G)$, according to the discrepancy $m$ between $\mu$ (the mass
of every $u_n$) and the mass of the weak limit $u$.
This can be applied to our minimization problem and, at least when $\mu$ is large, it reveals
that minimizing sequences are in fact strongly compact in $L^2$ and converge to a minimizer.
We point out that, beside the present application, Theorem~\ref{propgen}
can be applied to any minimization problem involving
the NLS energy on graphs, e.g. in the investigation of ground states (see Remark~\ref{rem23}).

At this stage a minimizer $u\in V_\mu$ satisfying \eqref{minpV} exists, yet one
would not expect $u$ to solve the Euler equation \eqref{euler} (i.e. to be
a critical point of the NLS energy subject to the mass constraint \eqref{mass}),
since
a generic variation of $u$ (albeit in a direction that preserves its mass)
might well violate
the additional condition that $u$ achieves its maximum on the edge $e$. In fact,
we can prove that, if $\mu$ is large enough, then $u$ achieves its maximum \emph{only on $e$}
(and away from branching points),
so that $u$ is strictly smaller elsewhere on $\G$ (Lemma~\ref{three}): in other words,
$u$ lies in the \emph{interior} of the additional constraint, which therefore becomes
stable under small perturbations of $u$. This, in turn, allows for arbitrary small
variations in any mass-preserving direction, which leads to the Euler equation \eqref{euler}
with a Lagrange multiplier $\lambda$ (just as if $u$ were a ground state, i.e. a minimizer
subject to the mass constraint only).

We point out that, in general, the bound states obtained in this way are not ground states,
and this is certainly so if $\G$ admits no ground state (see Example~\ref{example1}).
More generally, if $\G$ has $k$ bounded edges and $\mu$ is large, by the above procedure we
obtain at least $k$ distinct bound states (Corollary~\ref{coroll}), and not all of them
are necessarily ground states (see Example~\ref{example2}).

As far as we know, the results contained in this paper are the first to establish the existence of (many)
positive bound states by variational techniques. In particular, the fact that such bound states arise as {\em local minimizers}
of the constrained energy functional has a relevant consequence as regards the dynamics described by the Schr\"odinger Equation with an additional focusing power term:
$$
i \partial_t u (t) = - \Delta u (t) - | u (t) |^{p-2} u (t),
$$
where the symbol $\Delta$ is to be interpreted as the operator acting as the second derivative on functions satisfying Kirchhoff conditions \eqref{kir} at every vertex. It is indeed well-known (\cite{gss}) that local minimizers of the constrained energy functional are {\em orbitally stable} so that, if $u$ is any bound state described in Theorem \ref{mainteo}, then for the set
$ \{ e^{i \theta} u, \, \theta \in [ 0, 2 \pi) \}$
a stability property in the sense of Lyapunov holds. This  is relevant because,
while for ground states orbital stability is guaranteed by classical results (\cite{cl}), for excited states (namely bound states that are {\em not} ground states) it is often an open issue.
Indeed,
existence of positive excited states was already proved for the tadpole graph in \cite{cfn}, but  
the problem of their stability was left untouched (nonetheless, our results do not apply to such states). On the other hand, in \cite{nps} many non-positive excited states of the tadpole graph were found to be unstable.

Finally, we remark that  it would
be interesting to investigate the existence of sign-changing bound states. This has been done in the papers \cite{st, st2} but only in the context of localized nonlinearities, and in the already cited \cite{cfn,nps} for the particular case of the tadpole graph. Moreover,
it is an open problem to establish the existence of bound states
without the assumption that
the mass $\mu$ be large enough.

The paper can be summarized as follows: in Section 2 we give some results that work as  preparation for the main theorems: in particular, in Theorem \ref{propgen} we give a new, short proof of the beaviour of weakly convergent sequences. The core of Section 3 is Theorem \ref{minesiste}, that establishes the existence of a solution to the double-constrained minimimum problem for a sufficiently large mass. Section 4 contains  Theorem \ref{mainteo}, stating that the solutions whose existence is ensured in Theorem \ref{minesiste} are actually bound states. A remarkable technical step is achieved by Lemma \ref{three}, where we prove that the constraint is actually an open subset of the hypersurface of constant mass. Finally, in Section 5 we provide further explanations and examples in order to make the scope of the results more precise.

\section{Some basic tools (old and new)}

The purpose of this section is twofold. First, we recall some known facts and inequalities
on metric graphs which will be used throughout the paper.  Second, we prove a new result
which, we believe, may be of interest in any minimization procedure for the NLS energy on
metric graphs.

\medskip

Two basic tools in our analysis are the Gagliardo--Nirenberg
inequality
\begin{equation}
\label{GN-universal}
\forall u\in H^1(\G)\qquad
\Vert u\Vert_{L^p(\G)}^p\leq
C \Vert u\Vert_{L^2(\G)}^{\frac p 2 + 1}
\Vert u'\Vert_{L^2(\G)}^{\frac p 2 - 1}
\end{equation}
for some constant
$C$ depending only on $p$, and the $L^\infty$ estimate
\begin{equation}
\label{interpol}
\Vert u\Vert_{L^\infty(\G)}^2 \leq
2 \Vert u\Vert_{L^2(\G)}
\Vert u'\Vert_{L^2(\G)}.
\end{equation}
Both are well known when $\G=\R^+$, and
by a rearrangement argument they are proved to be valid on any
noncompact metric graph $\G$, with the same constant as on $\R^+$
(see \cite{ast2,T} for more details).

In the following we will need to
compare the achievable  energy levels on $\G$ with the ground-state energy
levels on the real line $\R$ and on the halfline $\R^+$, exploiting the fact that
in the prototypical cases
where $\G=\R$ or $\G=\R^+$ everything is known explicitly.

When $\G=\R$ (and for fixed $p\in (2,6)$) the ground states of mass $\mu$
are called \emph{solitons}, and we refer to \cite{cazenave} for more details.
Here, we just record that they
achieve a negative energy level given by
\begin{equation}
  \label{ensol}
  \min_{u\in H^1_\mu(\R)} E(u,\G)=\,-\,\theta_p \,\mu^{2\beta+1},
\end{equation}
for some constant $\theta_p>0$ and an exponent $\beta$ given by
\begin{equation}
  \label{defbeta}
\beta=\frac {p-2} {6-p}>0.
\end{equation}
Similarly, when $\G=\R^+$ the unique ground state
of mass $\mu$
is a ``half soliton'' (i.e. the restriction to $\R^+$ of
a soliton of mass $2\mu$ on $\R$), with an energy level  given by
\begin{equation}
  \label{enmsol}
  \min_{u\in H^1_\mu(\R^+)} E(u,\G)=\,-\,2^{2\beta}\theta_p \,\mu^{2\beta+1}
\end{equation}
with the same $\theta_p$ and $\beta$ as before.
We point out
that these two energy levels on $\R^+$ and $\R$
provide universal bounds for  the ground-state energy level on any noncompact
metric graph $\G$. More precisely,
it was proved in \cite{ast} that
\begin{equation}
\label{trapped2}
\,-\,2^{2\beta}\,\theta_p\, \mu^{2\beta+1}\,\leq
\inf_{u\in H^1_\mu(\G)} E(u,\G)\,\leq\, 
\,-\,\theta_p\, \mu^{2\beta+1}.
\end{equation}
Later on we will need the following lemma, which gives a lower bound
to the energy $E(w,\G)$ of a function, in terms of the number $N$ of its preimages.
When $N=2$ the result, based on a rearrangement technique,
is well known (see e.g. \cite{ast}, Proposition~3.1), while the general
case is a little more tricky, though in the same spirit.

\begin{lemma}
\label{ge3}
Let $w\in H^1(\G)$  be a nonnegative function
having at least $N$ preimages ($N\geq 1$) for almost every value, i.e.
\begin{equation}
\label{hpN}
\#\bigl\{ x\in \G\,|\,\, w(x)=t\bigr\} \geq N\quad\text{for a.e. $t\in (0,\|w\|_{L^\infty(\G)})$}.
\end{equation}
 Then
\begin{equation}
\label{toobig}
E(w,\G) \ge -\theta_p\left(\frac2N\right)^{2\beta}\nu^{2\beta+1},\qquad
\nu:=\Vert w\Vert_{L^2(\G)}^2.
\end{equation}
\end{lemma}

\begin{proof}  Let $w^*\in H^1(\R^+)$ denote the decreasing rearrangement of $w$
on the positive halfline. Regardless of \eqref{hpN}, it is well known (see \cite{ast}) that
\[
\Vert w^*\Vert_{L^r(\R^+)} =
\Vert w\Vert_{L^r(\G)}
\quad \forall r\in [1,+\infty],\qquad
\Vert (w^*)'\Vert_{L^2(\R^+)} \leq
\Vert w'\Vert_{L^2(\G)}.
\]
If, however, one has \eqref{hpN} (which is nontrivial only if $N>1$), then the last inequality
takes the stronger form
\begin{equation}
\label{strongPS}
\Vert (w^*)'\Vert_{L^2(\R^+)}^2 \leq\frac 1 {N^2}
\Vert w'\Vert_{L^2(\G)}^2,
\end{equation}
as pointed out in \cite{Duff}. This, however,
follows immediately from a direct inspection of the proof of the Polya--Szeg\H{o}
inequality (see \cite{kawohl} or \cite{friedlander} for the metric graph setting),
which is based on a slicing argument: at each level $t$ in the range of $w$,
letting $N(t)=\# w^{-1}(t)$,
at a certain stage one uses
the minorization $N(t)^2\geq 1$
(which is trivial,  but optimal for a generic $w$).
In fact, the proof yields the sharp bound \eqref{strongPS},
if one exploits \eqref{hpN} and uses $N(t)^2\geq N^2$ for a.e. $t$.

Now define $v\in H^1(\R^+)$ by letting $v(x) = w^*(Nx)$, and observe that
\[
\Vert v\Vert_{L^p(\R^+)}^p
 = \frac1N
 \Vert w^*\Vert_{L^p(\R^+)}^p
= \frac1N
\Vert w\Vert_{L^p(\G)}^p,
\]
while
\[
\Vert v'\Vert_{L^2(\R^+)}^2
 = N
 \Vert (w^*)'\Vert_{L^2(\R^+)}^2
\leq \frac1N
\Vert w'\Vert_{L^2(\G)}^2.
\]
Therefore, we have
\begin{align*}
E(w,\G)
=\frac 1 2 \Vert w'\Vert_{L^2(\G)}^2-\frac 1 p
\Vert w\Vert_{L^p(\G)}^p
\geq
\frac N 2 \Vert v'\Vert_{L^2(\R^+)}^2-\frac N p
\Vert v\Vert_{L^p(\R^+)}^p=
N\cdot E(v,\R^+).
\end{align*}
On the other hand,
\[
\Vert v\Vert_{L^2(\R^+)}^2
 = \frac1N
 \Vert w^*\Vert_{L^2(\R^+)}^2
= \frac1N
\Vert w\Vert_{L^2(\G)}^2=\frac \nu N
\]
so that we can use \eqref{ensol} (with $\mu=\nu$) to estimate
\[
E(v,\R^+)\geq - 2^{2\beta}\theta_p \left(\frac\nu N\right)^{2\beta+1},
\]
and \eqref{toobig} is established.
\end{proof}

\begin{remark}
As suggested by the proof,  the bound in \eqref{toobig} can be written as $-N \theta_p 2^{2\beta} (\nu/N)^{2\beta+1}$.
  Due to  \eqref{enmsol}, we can interpret this quantity as $N$ times
 the energy of a half-soliton  of mass $\nu/N$ on $\R^+$. In particular, when $N=1$
 this is the energy level in \eqref{enmsol}, while when $N=2$ this is
 the energy level in \eqref{ensol}. When $N=3$ (or bigger), this is the lowest
 energy level achievable on the star-graph $S_N$ (the union of $N$ half-lines joined
 at the origin), among functions which respect the $N$-fold symmetry of the graph (see
 \cite{acfn2}).
\end{remark}

Now we
show that if a sequence of functions tends to vanish locally on $\G$,
while dispersing a certain mass $m$ at infinity along the halflines of $\G$,
then the limit energy level is not smaller
than the energy level of a soliton of mass $m$.

\begin{lemma}[energy loss at infinity]\label{eli} Let $\G$ be a noncompact metric graph,
and assume that a sequence $\{v_n\}\subset H^1(\G)$ satisfies
\[
v_n\to 0\quad\text{in $L^\infty_{\text{loc}}(\G)$,}\qquad
\lim_n \Vert v_n\Vert_{L^2(\G)}^2=m\in (0,+\infty).
\]
Then
\[
\liminf_n E(v_n,\G)\geq -\theta_p m^{2\beta+1}.
\]
\end{lemma}
\begin{proof}
Replacing $v_n$ with $|v_n|$, we may assume that $v_n\geq 0$.
Then, setting $M_n:=\max_\G v_n$, two cases may occur according to where $M_n$ is attained.

\noindent\emph{(i)} If $M_n$ is attained in the interior of a halfline $\HH$
of $\G$, putting a coordinate $x\in [0,+\infty)$ on $\HH$, and setting $\delta_n=v_n(0)$,
we define on $\HH$ the modified function
\[
\widehat v_n(x)=\begin{cases}
  |x-\delta_n|&\text{if $x\in [0,2\delta_n]$,}\\
  v_n(x-2\delta_n) &\text{if $x>2\delta_n$,}
\end{cases}
\]
and observe that $\widehat v_n\in H^1(\HH)$ and $\widehat v_n(0)=v_n(0)=\delta_n$. Then we
set $\widehat v_n(x)=v_n(x)$ for $x\in\G\setminus \HH$, so that $\widehat v_n\in H^1(\G)$.
Now, considering that $\widehat v_n(\delta_n)=0$, that $\widehat v_n(x)\to 0$ as $x\to+\infty$ on $\HH$,
and that $\widehat v_n$ attains its maximum $M_n$ somewhere along $\HH$, we have
\[
\# \left\{ x\,|\,\, \widehat v_n(x)=t\right\}\geq 2\qquad
\forall t\in (0,M_n).
\]
Therefore, applying Lemma \ref{ge3} with $N=2$, we have
\[
E(\widehat v_n,\G)\geq
-\theta_p m_n^{2\beta+1},\quad
m_n:=\Vert \widehat v_n\Vert_{L^2(\G)}^2.
\]
But since
$v_n\to 0$ in $L^\infty_{\text{loc}}(\G)$, we have $\delta_n\to 0$, so that
$m_n=m+o(1)$ as $n\to\infty$. Similarly, also
$E(v_n,\G)=E(\widehat v_n,\G)+o(1)$,
so that the last inequality yields
\[
E(v_n,\G)\geq
-\theta_p m^{2\beta+1}+o(1)\qquad\text{as $n\to\infty.$}
\]
\noindent\emph{(ii)} If $M_n$ is \emph{not} attained in the interior of any halfline
of $\G$, then it is attained in the complementary set, which is
a compact set independent of $n$: since
$v_n\to 0$ in $L^\infty_{\text{loc}}(\G)$,  we must have $M_n\to 0$. But then a stronger
inequality holds, namely
\[
E(v_n,\G)\geq -\,\frac 1 p \int_\G |v_n(x)|^p\,dx
\geq -\,\frac {M_n^{p-2}} p \int_\G |v_n(x)|^2\,dx
\]
which is $o(1)$ as $n\to \infty$. Thus, the claim is proved in either case.
\end{proof}

\begin{theorem}\label{propgen}
Let $\{u_n\}\subset \Hmu(\G)$ be a sequence such that $u_n\weak u$ in $H^1(\G)$ and a.e. on $\G$, and
let
\begin{equation}
\label{defm}
m:=\mu-\Vert u\Vert_{L^2(\G)}^2\in [0,\mu]
\end{equation}
be the loss of mass in the limit.
Then, letting
$\Lambda:=\liminf_n E(u_n,\G)$,
one of the following  alternatives occurs, depending on the value of $m$:
\begin{itemize}
  \item[(i)] $m=0$. Then $u_n\to u$ \emph{strongly} in $L^2(\G)\cap L^p(\G)$, $u\in H^1_\mu(\G)$
  and $E(u,\G)\leq \Lambda$.
\item[(ii)] $0<m<\mu$. Then
\begin{equation}\label{stimaL}
\Lambda>\min \left\{  -\theta_p \mu m^{2\beta},\,
E(w,\G)\right\},
\end{equation}
  where $w\in H^1_\mu(\G)$ is the renormalized limit
\begin{equation}
\label{defw}
w(x)=\sqrt{\frac \mu{\mu-m}}\,\, u(x)\quad \forall x\in \G.
\end{equation}
\item[(iii)]  $m=\mu$. Then $u\equiv 0$ and $\Lambda\geq -\theta_p \mu^{2\beta+1}$.
\end{itemize}
\end{theorem}

\begin{proof}
If $m=0$, then $u_n\to u$ strongly in $L^2(\G)$ and
therefore $u\in H^1_\mu(\G)$.
Moreover, since $u_n$ is bounded in $H^1(\G)$
hence in $L^\infty(\G)$, $u_n\to u$ strongly also in $L^p(\G)$. Then
$E(u,\G)\leq\Lambda$ by semicontinuity of the kinetic term.

If, on the other hand, $m=\mu$, then $u\equiv 0$ so that $u_n\to 0$ in
$L^\infty_{\text{loc}}(\G)$. In this case, $\Lambda\geq -\theta_p\mu^{2\beta+1}$
follows from Lemma \ref{eli}, applied with $v_n=u_n$ and $m=\mu$.

Finally, assume $0<m<\mu$. According to
the Brezis--Lieb Lemma (\cite{bl}) we may split
\begin{equation}
\label{s20}
E(u_n,\G) = E(u_n -u,\G) + E(u,\G) + o(1)\quad\text{as $n\to\infty$.}
\end{equation}
Now, from $u_n\weak u$ in $L^2(\G)$ we have
\[
\Vert u_n-u\Vert_{L^2(\G)}^2=
\Vert u_n\Vert_{L^2(\G)}^2+
\Vert u\Vert_{L^2(\G)}^2
-2 \langle u_n,u\rangle_{L^2(\G)}
\to \mu-\Vert u\Vert_{L^2(\G)}^2 =m
\]
as $n\to\infty$,
hence Lemma \ref{eli} applied with $v_n=u_n-u$ yields
\[
\liminf_n E(u_n-u,\G)\geq
-\theta_p m^{2\beta+1}.
\]
Thus, taking the liminf in \eqref{s20}, we find
\[
\Lambda\geq
-\theta_p m^{2\beta+1}
+E(u,\G).
\]
On the other hand, since $p>2$ and $u\not\equiv 0$, we have
\[
E(w,\G)
=\frac 1 2\,\frac\mu {\mu-m}
\Vert u'\Vert_{L^2(\G)}^2
-\frac 1 p
\left(\frac\mu {\mu-m}\right)^{\frac p 2}
\Vert u\Vert_{L^p(\G)}^p <
\frac\mu {\mu-m} E(u,\G)
\]
and thus from the previous inequality we obtain
\[
\Lambda >
-\theta_p m^{2\beta+1}
+\left(1-\frac m\mu\right) E(w,\G)=-\left(\frac m\mu\right)\theta_p \mu m^{2\beta}+
\left(1-\frac m\mu\right) E(w,\G),
\]
and \eqref{stimaL} follows.
\end{proof}

\begin{remark}\label{rem23}
  The typical application of the previous theorem is when $u_n$ is a minimizing
  sequence for some variational problem, i.e.
$
  \Lambda=\inf_{w \in V} E(w,\G)
$
  where $V\subseteq H^1_\mu(\G)$ is a certain class of functions of mass $\mu$.
  For instance, in \cite{ast2} it was proved that, when $V=H^1_\mu(\G)$, the condition
  $
  \Lambda < -\theta_p \mu^{2\beta+1}
$ (i.e. the strict inequality in the upper bound in \eqref{trapped2})
is sufficient for the existence of a ground state of mass $\mu$ on $\G$.
Now  Theorem~\ref{propgen} yields a short proof of this result. Indeed, the assumption
on $\Lambda$ immediately rules out case (iii), and \eqref{stimaL} is also impossible
since $w$ is an admissible competitor: then case (i) must occur, and the limit $u$ is a ground state.
\end{remark}

\section{Energy minimization with a localized maximum}

In this section we investigate the existence of a function with least energy, among
all functions of prescribed mass that achieve their maximum on a prescribed bounded edge
of $\G$.

Throughout,  $\G$ denotes a noncompact, connected metric graph, having at least
one bounded edge (if every edge of $\G$ is a halfline, the existence of
bound states has already been investigated in \cite{acfn2}).
We also denote by $e$ a \emph{fixed} bounded edge of $\G$, with vertices $\vv_1$ and
$\vv_2$, with the only assumption that $e$ is maximal,
i.e. $\deg(\vv_1)\not=2$ and $\deg(\vv_2)\not=2$ (there is no loss of generality
in this assumption, since any vertex $\vv$ of degree two can a priori be \emph{eliminated}
from any metric graph, by melting the two edges incident at $\vv$ into a single
edge). Observe
that the two conditions $\deg(\vv_1)=1$ and $\deg(\vv_2)=1$ cannot
be satisfied simultaneously,
otherwise
$e$ would be an isolated edge and $\G$, being connected, would be compact.
Thus, swapping if necessary $\vv_1$ with $\vv_2$, we can concretely
assume that
\begin{equation}
  \label{assvv} \deg(\vv_1)\not=2,\qquad
  \deg(\vv_2)\geq 3.
\end{equation}
If $\deg(\vv_1)=1$ then $e$ is called a \emph{terminal edge} of $\G$. Finally, we do not
exclude the case where $e$ forms a \emph{self loop}: in other words, it may well
happen that
$\vv_1=\vv_2$.

We shall focus our attention on the class of functions
\begin{equation}
\label{defV}
V=
\left\{u\in H^1(\G)\,:\;  \|u\|_{L^\infty(e)} = \|u\|_{L^\infty(\G)}\right\},
\end{equation}
characterized by the fact that their $L^\infty$ norm is achieved on the edge $e$,
and, for every $\mu>0$, on the subclass
\begin{equation}
\label{defVmu}
\Vmu = V \cap \Hmu(\G)
\end{equation}
of functions satisfying the additional mass constraint \eqref{mass}.

\begin{remark}\label{remVcl} The set $V$ is closed in the weak topology of $H^1(\G)$. Indeed,
if $u_n\in V$ and $u_n \weak u$ in $H^1(\G)$, then  by semicontinuity
\begin{align*}
\|u\|_{L^\infty(\G)} \leq
\liminf_{n}
\|u_n\|_{L^\infty(\G)}
=\liminf_{n}
\|u_n\|_{L^\infty(e)}=
\|u\|_{L^\infty(e)}
\end{align*}
(the last equality follows since $u_n \to u$ uniformly on $e$).
\end{remark}

We shall study the following minimization problem: find $u \in V_\mu$ such that
\begin{equation}
\label{minpV}
E(u,\G) = \inf_{v\in V_\mu}E(v,\G).
\end{equation}
The first crucial step is to estimate the infimum in \eqref{minpV} in terms
of $\mu$. Now we show that, for large $\mu$, this infimum is very close to
the ground-state energy level on $\R$ given by \eqref{ensol}. Moreover, if $e$
is a terminal edge, then the bound can be improved, and the infimum approaches
the ground-state energy level on $\R^+$ given by \eqref{enmsol}.

\begin{lemma}
\label{levelest}
For every $\varepsilon>0$ there exists $\mu_\varepsilon$ (depending only on $\eps$
and on the length of the edge $e$) such that, for all $\mu \ge \mu_\varepsilon$,
\begin{equation}
\label{quasisol}
\inf_{v\in \Vmu} E(v,\G) \le -\theta_p (1-\varepsilon)\mu^{2\beta+1}.
\end{equation}
Moreover, if $e$ is a terminal edge, then the previous inequality can be replaced with
\begin{equation}
\label{quasimezzosol}
\inf_{v\in \Vmu} E(v,\G) \le -\theta_p (1-\varepsilon)2^{2\beta}\mu^{2\beta+1}.
\end{equation}
\end{lemma}

\begin{proof} Let $\phi\in H^1(\R)$ denote the soliton of unitary mass centered at the origin.
Since $\phi$
solves the minimization problem \eqref{ensol} when $\mu=1$,
$E(\phi,\R)=-\theta_p$ and hence,
given $\eps>0$, by standard density arguments there exists $\phi_\eps \in H^1(\R)$
with \emph{compact support}, such that
\[
E(\phi_\eps,\R)\leq -(1-\eps)\theta_p,\qquad \Vert \phi_\eps\Vert_{L^2(\R)}^2=1.
\]
Then, recalling \eqref{defbeta}, for $\mu>0$ we define the following rescaled version of $\phi_\eps$
\[
v_\mu(x)=\mu^\alpha \phi_\eps(\mu^\beta x),\quad x\in\R,\quad
\alpha=\frac 2{6-p}
\]
(this is the natural scaling for the NLS energy, see  Remark 2.2 in \cite{ast2}), and one can check that
\[
\Vert v_\mu \Vert_{L^2(\R)}^2=
\mu \Vert \phi_\eps\Vert_{L^2(\R)}^2=\mu,\qquad
E(v_\mu,\R)=\mu^{2\beta+1} E(\phi_\eps,\R)\leq -(1-\eps)\theta_p\mu^{2\beta+1}.
\]
Now, if  $\mu$ is large enough, the diameter of the support of $v_\mu$
becomes smaller than the length of the edge $e$: thus, we can fit $v_\mu$ on $e$
(so that it vanishes at $\vv_1$ and $\vv_2$)
and, setting it equal to zero on $\G\setminus e$, we may regard $v_\mu$ as a function
in $H^1_\mu(\G)$. This produces an admissible competitor in \eqref{quasisol} and, since
in the last inequality we can interpret $E(v_\mu,\R)$ as $E(v_\mu,\G)$, \eqref{quasisol}
is proved.

Finally, assume $e$ is a terminal edge (according to \eqref{assvv}, $e$ is attached to
$\G$ at $\vv_2$ while $\vv_1$ is the tip of the edge, with $\deg(\vv_1)=1$). The previous
argument remains valid, but a variant of it gives a better bound.
Indeed, now
one can start with $v_{2\mu}\in H^1(\R)$ instead
of $v_\mu$,
and exploit the fact that $v_{2\mu}(-x)=v_{2\mu}(x)$, so that on $\R^+$ one has
\[
\Vert v_{2\mu} \Vert_{L^2(\R^+)}^2=
\frac 1 2  \Vert v_{2\mu}\Vert_{L^2(\R)}^2=\mu,\qquad
E(v_{2\mu},\R^+)=\frac 1 2 E(v_{2\mu},\R)\leq -(1-\eps)\theta_p2^{2\beta}\mu^{2\beta+1}
\]
(the same bound as in \eqref{quasimezzosol}). Then, for large $\mu$, one can
\emph{restrict}  $v_{2\mu}$ to $\R^+$, fit it on $e$ so that it vanishes at $\vv_2$,
and set it equal to zero
on $\G\setminus e$ as before (but now $v_{2\mu}$, once fitted on $\G$, will satisfy a Neumann condition at $\vv_1$,
rather than  Dirichlet).

\end{proof}

\begin{theorem}
\label{minesiste}
There exists a mass threshold $\overline\mu$ such that,
for every $\mu\geq\overline\mu$, the minimization problem \eqref{minpV}
has a solution $u\in H^1_\mu(\G)$.
\end{theorem}

\begin{proof}
The mass threshold will be obtained from Lemma \ref{levelest}
letting $\overline\mu:=\mu_\eps$,
for a suitable choice of $\eps\in (0,1/2)$ to be specified later.
Now, given $\mu\geq\mu_\eps$,
let $\{u_n\}\subset V_\mu$ be a minimizing sequence for problem \eqref{minpV}, that is
\[
  \lim_{n\to\infty} E(u_n,\G)=\infV:=\inf_{v\in V_\mu} E(v,\G).
\]
Since $\eps<1/2$, \eqref{quasisol} entails $\Lambda<-\theta_p \mu^{2\beta+1}/2$, so we may assume
that $E(u_n,\G)\leq-\theta_p \mu^{2\beta+1}/2$ for every $n$, which, using \eqref{NLSe} and rearranging terms,
can be written as
\begin{equation}
  \label{sl34}
\frac{\theta_p}2 \mu^{2\beta+1} +
\frac 1 2 \Vert u_n'\Vert_{L^2(\G)}^2\leq\frac 1 p \Vert u_n\Vert_{L^p(\G)}^p.
\end{equation}
This will be used in two ways. First, ignoring $u_n'$ and using $u_n\in V_\mu$, we have
\[
\frac{p\theta_p}2 \mu^{2\beta+1}
\leq \Vert u_n\Vert_{L^p(\G)}^p\leq  \Vert u_n\Vert_{L^\infty(\G)}^{p-2}
\Vert u_n\Vert_{L^2(\G)}^2
= \mu \Vert u_n\Vert_{L^\infty(e)}^{p-2},
\]
which (recalling \eqref{defbeta}) we may record for future usage as
\begin{equation}
  \label{stimaLI}
  C_p \mu^{\beta+1}\leq \Vert u_n\Vert_{L^\infty(e)}^2.
\end{equation}
Second, if we drop the first summand in \eqref{sl34}, and use \eqref{GN-universal}
combined with $\Vert u_n\Vert_{L^2}^2=\mu$ to estimate the right hand side,
we obtain
\begin{equation}\label{stimaK}
\Vert u_n'\Vert_{L^2(\G)}^2\leq C_p \mu^{2\beta +1}.
\end{equation}
Therefore $\{u_n\}$ is bounded in $H^1(\G)$,
and we may assume that $u_n\weak u$ in $H^1(\G)$ and $u_n\to u$
pointwise a.e. on $\G$,
for some $u\in H^1(\G)$. Thus Theorem \ref{propgen}
applies, and we shall prove that case (i) occurs.

Since
$u_n\to u$ uniformly on the edge $e$, recalling \eqref{stimaLI} we have
\[
\Vert u\Vert_{L^\infty(\G)}^2
\geq
\Vert u\Vert_{L^\infty(e)}^2
=
\lim_n
\Vert u_n\Vert_{L^\infty(e)}^2
\geq
C_p \mu^{\beta+1}.
\]

In particular, $u\not\equiv 0$, and this (defining $m$ as in \eqref{defm})
rules out case (iii) of Theorem~\ref{propgen}.
Moreover, using \eqref{stimaLI}, \eqref{interpol}, \eqref{defm} and \eqref{stimaK}, we have
\begin{align*}
C_p \mu^{\beta+1} &\leq
\Vert u\Vert_{L^\infty(\G)}^2
\leq
2\Vert u\Vert_{L^2(\G)}
\Vert u'\Vert_{L^2(\G)}
=
2\sqrt{\mu-m\,}\,\,\Vert u'\Vert_{L^2(\G)}
\\
&\leq
2\sqrt{\mu-m\,}\,\,\liminf_n \Vert u_n'\Vert_{L^2(\G)}
\leq
C_p'\sqrt{\mu-m\,}
\,\, \mu^{\beta+\frac 1 2}.
\end{align*}
Thus we find $\mu-m\geq \delta_p \mu$ for some universal constant $\delta_p>0$,
that is
\begin{equation}
\label{stimam}
m\leq (1-\delta_p)\mu.
\end{equation}

Now we now show that \eqref{stimaL} is violated. Indeed,
since $u\in V$
by Remark \ref{remVcl}, the renormalized function $w$ defined in \eqref{defw} belongs to $V_\mu$,
and therefore $\Lambda\leq E(w,\G)$, so that \eqref{stimaL} simplifies to
$\Lambda> -\theta_p\mu \,m^{2\beta}$,
which combined with \eqref{quasisol} gives
\[
-\theta_p (1-\eps)\mu^{2\beta+1}>-\theta_p\mu \,m^{2\beta}.
\]
But this is incompatible with \eqref{stimam}, if $\eps$ is small enough.
Summing up, for small $\eps$,
we are in case (i) of Theorem~\ref{propgen}, that is, $u$
is a solution to problem \eqref{minpV}.
\end{proof}

\section{Bound states}

Now we prove that, for large $\mu$, every solution to the minimization
problem \eqref{minpV} achieves its maximum \emph{only} on the edge $e$, and
nowhere else.

\begin{remark}\label{rempos} If $u\in V_\mu$ is a solution to \eqref{minpV}, then
also $|u|$ is a solution, since $|u|\in V_\mu$ and $E(|u|,\G)=E(u,\G)$. In fact,
as we will prove
later, for large $\mu$ every solution is either strictly positive
o strictly negative on $\G$.
\end{remark}

\begin{lemma}
\label{three} Let $u \in \Vmu$ be a solution to \eqref{minpV}, found
according to Theorem \ref{minesiste}.
If $\mu$ is large enough (depending only on $p$ and on the graph $\G$), then
\begin{equation}
\label{toprove}
\|u\|_{L^\infty(e)} > \|u\|_{L^\infty(\G\setminus e)}.
\end{equation}
\end{lemma}

\begin{proof} Considering $|u|$ instead of $u$,
by Remark~\ref{rempos} we may assume that $u\geq 0$ (note that
\eqref{toprove} depends only on $|u|$).
Let $M= \|u\|_{L^\infty(e)}$. Since $u\in\Vmu$, we clearly have
$M \ge \|u\|_{L^\infty(\G\setminus e)}$.
To prove the claim, we assume that
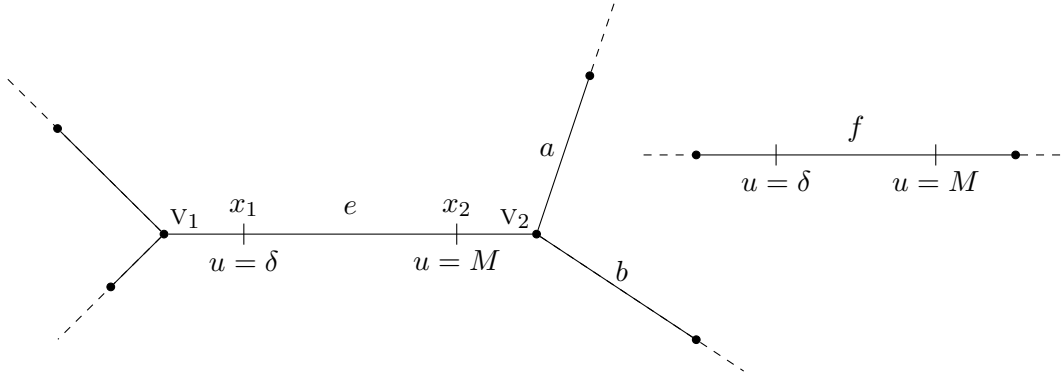
\begin{equation}
\label{assurda}
M= \|u\|_{L^\infty(\G\setminus e)}
\end{equation}
and we obtain a contradiction, if $\mu$ is large enough.
To start with, we certainly assume $\mu\geq \overline{\mu}$ (the mass threshold
of Theorem~\ref{minesiste}) but, throughout the proof, a possibly larger mass threshold
will arise.

The proof is divided into two steps.

\medskip

\newcommand\BB{\mathcal B}
\noindent\emph{Step 1). } Denoting by $\BB$ the set of all the bounded edges of $\G$, and
letting
\begin{equation}
\label{defdelta}
\delta:=\max_{h\in \BB} \,\min_{x\in h} u(x)\qquad\quad (0\leq\delta\leq M),
\end{equation}
we will prove a nontrivial lower bound on the number of preimages in $u^{-1}(t)$, for every value
$t\in (\delta,M)$ achieved by $u$ on $\G$. The precise lower bound depends
on whether $e$ is a terminal edge or not. More precisely, recalling \eqref{assvv}, we will show that
\begin{equation}
  \label{caso1}
\deg (\vv_1)=1\qquad\Rightarrow\qquad
\#\left\{x\in\G\,|\,\,u(x)=t\right\} \ge 2 \quad\text{for every  $t\in (\delta,M)$,}
\end{equation}
while
\begin{equation}
  \label{caso3}
\deg (\vv_1)\geq 3\qquad\Rightarrow\qquad
\#\left\{x\in\G\,|\,\,u(x)=t\right\} \ge 3 \quad\text{for every  $t\in (\delta,M)$}
\end{equation}
We point out that $\delta\geq 0$ since $u\geq 0$. Moreover we may assume $\delta<M$, otherwise the claims are trivial.

Now, \eqref{assurda} means  that $u$ (which already attains its maximum $M$ on $e$)
achieves the value $M$
also
on \emph{another}
edge $f$ (possibly at $\vv_1$ or $\vv_2$, if $e$ and $f$ share a vertex).
Moreover, since $u$ is continuous, in each edge $g$ (bounded or not) where $u$ attains $M$,
$u$ also attains \emph{every intermediate value} $t\in (\delta,M)$: indeed, if $g$ is a bounded edge
this follows immediately from \eqref{defdelta}, while if $g$ is a halfline, then $u(x)\to 0$
as $x$ tends to infinity along $g$.
In particular, for every $t\in (\delta,M)$, $u$
has at least two preimages on $\G$ (one on $e$ and one on $f$), and \eqref{caso1} is
established (regardless of $\deg(\vv_1)$).

Now assume that  $e$ is not a terminal edge, i.e.
$\deg (\vv_1) \ge 3$, and consider two points $x_1,x_2\in e$ with
$u(x_1)=\delta$ and $u(x_2)=M$.
By swapping (if necessary) $\vv_1$ and $\vv_2$,
we may assume that, on the edge $e$,
$x_2$ lies between $x_1$ and
$\vv_2$, as shown in the picture (of course, we do not exclude that $x_2=\vv_2$
or $x_1=\vv_1$).

\begin{figure}
\begin{tikzpicture}[xscale= 0.7,yscale=0.7] %%% un terminal edge
\node at (-2,2) [nodo] (-22) {};
\node at (-1,-1) [nodo] (-11) {};
\node at (0,0) [nodo] (00) {};
\node at (7,0) [nodo] (30) {};
\node at (8,3) [nodo] (43) {};
\node at (10,-2) [nodo] (-62) {};
\node at (10,1.5) [nodo] (72) {};
\node at (16,1.5) [nodo] (92) {};

\draw[-] (-22)--(00);
\draw[-] (-11)--(00);
\draw[-] (00)--(30);
\draw[-] (30)--(43);
\draw[-] (30)--(-62);
\draw[-] (72)--(92);

\node at (9,1.5) [infinito]  (x) {};
\node at (17,1.5) [infinito]  (y) {};
\draw[dashed] (x)--(72);
\draw[dashed] (92)--(y);

\node at (8.5,4.5) [infinito] (ad) {};
\draw[dashed] (43)--(ad);

\node at (-3,3) [infinito] (as) {};
\draw[dashed] (00)--(as);

\node at (-2,-2) [infinito] (bs) {};
\draw[dashed] (00)--(bs);

\node at (10.9,-2.6) [infinito] (bd) {};
\draw[dashed] (30)--(bd);

\draw [-]  (5.5,-.2)--(5.5,.2);
\node at (5.5,-.5) [infinito] (yy) {$u=M$};
\node at (5.5,.5) [infinito]  {$x_2$};

\draw [-]  (1.5,-.2)--(1.5,.2);
\node at (1.5,.5) [infinito]  {$x_1$};
\node at (1.5,-.5) [infinito] (M1) {$u=\delta$};
\draw [-]  (11.5,1.3)--(11.5,1.7);
\node at (11.5,1.0) [infinito] (zz) {$u=\delta$};

\draw [-]  (14.5,1.3)--(14.5,1.7);
\node at (14.5,1.0) [infinito] (M2) {$u=M$};

\node at (3.5,.5) [infinito]  (e) {$e$};
\node at (13,2) [infinito]  (effe) {$f$};
\node at (7.2,1.6) [infinito]  (a) {$a$};
\node at (8.6,-.7) [infinito]  (b) {$b$};
\node at (.4,.3) [infinito]  (A) {$\vv_1$};
\node at (6.6,.3) [infinito]  (B) {$\vv_2$};
\end{tikzpicture}
\end{figure}
Let $a,b$ denote two edges (other than $e$) emanating from $\vv_2$:
if $a\not=f$ then, for every value $t\in (\delta,M)$, a \emph{third preimage}
can be found (in addition to those already found on $e$ in the interval $[x_1,x_2]$ and,
similarly, on $f$)
by travelling on $e$ from $x_2$ (where $u=M$) to $\vv_2$ and,
if $u(\vv_2)>t$, travelling further along $a$, until $u=t$
(if $a$ is a bounded edge then, somewhere on $a$, $u=t$ by \eqref{defdelta}, while
if $a$ is a halfline, then $u$ achieves $t$ since
$u(\vv_2)>t$ while $u\to 0$ at infinity along $a$). Of course, the same can be done if $b\not=f$.
Finally, if $f=a=b$, then $f$ is necessarily a bounded edge, that
forms a self-loop attached at $\vv_2$: since on this loop $u$ attains the values
$M$ and $\delta$, every intermediate value $t\in (\delta,M)$ is attained  \emph{at least twice}
on $f$, so that in this case a third preimage can be found on $f$ itself.
Therefore, also \eqref{caso3} is established.

\medskip

\newcommand\Gd{\G'}

\noindent\emph{Step 2) } Now, in order to apply Lemma \ref{ge3}, we would like
the inequalities in \eqref{caso1} and \eqref{caso3} to be valid for every $t\in (0,M)$,
without the restriction $t>\delta$. This can be achieved by \emph{extending}  $u$
to a graph $\Gd$ larger than $\G$, as follows (if by any chance $\delta=0$, then no extension is
necessary: our argument, however, formally works also when $\delta=0$).

Given $\lambda>0$ (to be specified later), let $\Gd$ be the graph obtained from $\G$, with the addition
of \emph{two extra edges} $e_1$ and $e_2$, both of length $\lambda$, and both
with one endpoint attached to $\G$ at $x_1$ (recall $x_1\in e$ and $u(x_1)=\delta$).
Putting a coordinate $s\in [0,\lambda]$ on each $e_i$
(so that $e_i$ is attached to $x_1$
at $s=0$), we extend $u$ to $\Gd$ letting
\[
u(s)=\frac\delta\lambda(\lambda-s),\quad \forall s\in e_1\sim [0,\lambda],\quad
\forall s\in e_2\sim [0,\lambda].
\]
Clearly, $u\in H^1(\Gd)$. Moreover, $u$ now achieves every value $t\in (0,\delta)$
no less than three times on $\Gd$: once on $e_1$, once on $e_2$, and at least once on the original
graph $\G$ (since $u\to 0$ along every halfline of $\G$). This, combined with \eqref{caso1}
and \eqref{caso3}, gives
\begin{align*}
\deg (\vv_1)=1\qquad\Rightarrow\qquad
\#\left\{x\in\Gd\,|\,\, u(x)=t\right\} \ge 2 \quad\text{for almost every  $t\in (0,M)$,}
\\
\deg (\vv_1)\geq 3\qquad\Rightarrow\qquad
\#\left\{x\in\Gd\,|\,\, u(x)=t\right\} \ge 3 \quad\text{for almost every  $t\in (0,M)$.}
\end{align*}
Now Lemma \eqref{ge3} can be applied to $u$  on $\Gd$  in either case
(choosing $N=2$ or $N=3$ according to $\deg(\vv_1)$),
thus obtaining
\begin{align}
  \label{c1}
\deg (\vv_1)=1\qquad&\Rightarrow\qquad
E( u,\Gd)\geq -\theta_p \,\nu^{2\beta+1},
\\
  \label{c3}
\deg (\vv_1)\geq 3\qquad&\Rightarrow\qquad
E(u,\Gd)\geq -\theta_p \left(\frac 2 3\right)^{2\beta}\nu^{2\beta+1}
\end{align}
where
\begin{equation}
\label{stimanu}
\nu:=\Vert u\Vert_{L^2(\Gd)}^2
=
\Vert u\Vert_{L^2(\G)}^2
+
2\frac {\delta^2}{\lambda^2}\int_0^\lambda (\lambda-s)^2\,ds
\leq
\mu+\lambda\delta^2.
\end{equation}
Moreover,
\[
E(u,\Gd)=
E(u,\G)+2 \, E(u,e_1) \leq E(u,\G)+2 \times \frac 1 2 \Vert u'\Vert_{L^2(e_1)}^2=
E(u,\G)+\frac {\delta^2}\lambda,
\]
so that $E(u,\G)\geq E(u,\Gd)-\delta^2/\lambda$. Plugging this inequality into
\eqref{c1} and \eqref{c3}, and then using
\eqref{stimanu}, yields
\begin{align}
  \label{cc1}
\deg (\vv_1)=1\qquad&\Rightarrow\qquad
E( u,\G)\geq -\theta_p \,(\mu+\lambda\delta^2)^{2\beta+1}-\frac{\delta^2}{\lambda},
\\
  \label{cc3}
\deg (\vv_1)\geq 3\qquad&\Rightarrow\qquad
E(u,\G)\geq -\theta_p \left(\frac 2 3\right)^{2\beta}
(\mu+\lambda\delta^2)^{2\beta+1}-\frac{\delta^2}{\lambda}.
\end{align}
Moreover, since $E(u,\G)$ coincides with the infimum in \eqref{quasisol} (and in
\eqref{quasimezzosol}), for every small $\eps>0$ (assuming that $\mu\geq \mu_\eps$)
we can use Lemma \eqref{levelest} to estimate $E(u,\G)$ from above, and plug
\eqref{quasisol} and \eqref{quasimezzosol}, respectively, into \eqref{cc3} and \eqref{cc1}.
Reversing signs, this gives
\begin{align*}
\deg (\vv_1)=1\qquad&\Rightarrow\qquad
\theta_p (1-\varepsilon)2^{2\beta}\mu^{2\beta+1}
\leq \theta_p \,(\mu+\lambda\delta^2)^{2\beta+1}+\frac{\delta^2}{\lambda},
\\
\deg (\vv_1)\geq 3\qquad&\Rightarrow\qquad
\theta_p (1-\varepsilon)\mu^{2\beta+1}
\leq \theta_p \left(\frac 2 3\right)^{2\beta}
(\mu+\lambda\delta^2)^{2\beta+1}+\frac{\delta^2}{\lambda}.
\end{align*}
Of these two estimates, the latter (being weaker) is valid in either
case: therefore, from now on, we will focus on the latter, regardless of $\deg(\vv_1)$.
Now, to estimate $\delta$, observe that for every bounded edge $h$ of $\G$
we have
\[
\left(\min_{x\in h} u(x)\right)^2 \leq
\frac 1 {\mathop{\rm length}(h)} \int_h u(x)^2\,dx\leq
\frac \mu\ell,
\]
where $\ell>0$
is the length of the shortest edge of $\G$ (i.e. a constant that depends
only on $\G$). Therefore, from \eqref{defdelta} we have
$\delta^2\leq \mu/\ell$, and hence
\[
\theta_p (1-\varepsilon)\mu^{2\beta+1}
\leq \theta_p \left(\frac 2 3\right)^{2\beta}
\mu^{2\beta+1}(1+\lambda/\ell)^{2\beta+1}+\frac{\mu}{\lambda\ell}.
\]
Now, regardless of $\mu$, we can fix $\eps$ and $\lambda$ small enough,
so that the coefficient of $\mu^{2\beta+1}$ on the left is \emph{strictly bigger}
than the corresponding coefficient on the right: then, a contradiction is
easily found
if $\mu$ is large enough.
\end{proof}

As a consequence of this lemma, we can prove that any solution $u\in V_\mu$
to \eqref{minpV} is in fact a bound state.

\begin{theorem}
\label{mainteo}
Let $u \in \Vmu$ be a solution to \eqref{minpV}, found
according to Theorem \ref{minesiste}.
If $\mu$ is large enough (depending only on $p$ and on the graph $\G$), then
$u$ is a bound state of mass $\mu$, i.e. $u$ satisfies \eqref{mass}
and \eqref{euler} for a suitable $\lambda>0$.
Moreover, either $u>0$ or $u<0$ on $\G$.
\end{theorem}

\begin{proof} Since $u\in \Vmu\subset H^1_\mu(\G)$, \eqref{mass}
is obvious.
Moreover, by Lemma \ref{three}, we may rely on \eqref{toprove} which,
due to the \emph{strict} inequality,  is stable
under perturbations of $u$ that are sufficiently small in the $L^\infty$ norm
on $\G$. But since $H^1(\G)\hookrightarrow L^\infty(\G)$, we see that
\eqref{toprove} is stable also under small perturbations of $u$
in the $H^1$ norm: in other words, recalling \eqref{defV} and \eqref{defVmu},
$u$ lies in the \emph{interior} of $V$,
so that $u$ is not only a global minimizer in $V_\mu$, but also a local minimizer
in $H^1_\mu(\G)$. Then $u$ is a critical point of the NLS energy subject
to \eqref{mass}, and \eqref{euler} (with $\lambda$ as a Lagrange multiplier)
follows by
standard arguments.

The fact that $\lambda>0$ is due to the non-compactness of $\G$ (i.e. to the presence
of at least one halfline): indeed, it is well known that any solution of
\eqref{euleroforte} in $L^2(\R^+)$
is (the restriction of) a  soliton, which forces $\lambda>0$.

Finally, by Remark~\ref{rempos}, all we have proved for $u$ is valid also for $|u|$,
including the analogues of \eqref{euleroforte} and \eqref{kir}. Then (since $|u|$ is
not identically zero) $|u|>0$ can be proved exactly as in \cite{ast} (see (iii) of
Proposition~3.3). This shows that $u$ does not vanish on $\G$, therefore either
$u>0$ or $u<0$ since $\G$ is connected.
\end{proof}

\begin{corollary}\label{coroll}
  Let $\G$ be a noncompact metric graph, having $k$ bounded edges.
  If $\mu$ is large enough, then $\G$ admits at least $k$ positive
  bound states of mass $\mu$.
\end{corollary}
\begin{proof}
Since the edge $e$ (that determines the set $V$ defined in \eqref{defV})
can be chosen in $k$ different ways, for large $\mu$ one obtains $k$ bound
states (and replacing each $u$ with $-u$, we may assume that each of them is positive).
Since each of them satisfies \eqref{toprove} relative
to the corresponding edge $e$, we see that these $k$ bound states are, in all
respects, distinct.
\end{proof}

\section{Some examples}
In this section we discuss some examples
of metric graphs, in order to illustrate how our results can be concretely applied.

In the pictures, the unbounded edges (halflines) of the graph can be
recognized at a glance, since their terminal edge 
(the ``point at infinity''
of the halfline) is labeled  with the symbol ``$\infty$''.
The bullet symbol ``$\bullet$'', in contrast, is used to
label the ordinary vertices of the
graph.

\begin{example}\label{example1} The following graph

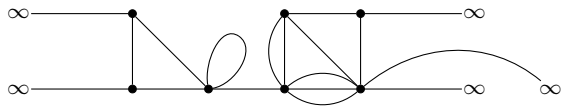
\begin{center}
\begin{tikzpicture}
\node at (0,0) [nodo] (1) {};
\node at (-1.5,0) [infinito]  (2){\infvert};
\node at (1,0) [nodo] (3) {};
\node at (0,1) [nodo] (4) {};
\node at (-1.5,1) [infinito] (5) {\infvert};
\node at (2,0) [nodo] (6) {};
\node at (3,0) [nodo] (7) {};
\node at (2,1) [nodo] (8) {};
\node at (3,1) [nodo] (9) {};
\node at (4.5,0) [infinito] (10) {\infvert};
\node at (5.5,0) [infinito] (11) {\infvert};
\node at (4.5,1) [infinito] (12) {\infvert};
\draw [-] (1) -- (2) ;
\draw [-] (1) -- (3);
\draw [-] (1) -- (4);
\draw [-] (3) -- (4);
\draw [-] (5) -- (4);
\draw [-] (3) -- (6);
\draw [-] (6) -- (7);
\draw [-] (6) to [out=-40,in=-140] (7);
\draw [-] (3) to [out=10,in=-35] (1.4,0.7);
\draw [-] (1.4,0.7) to [out=145,in=100] (3);
\draw [-] (6) to [out=40,in=140] (7);
\draw [-] (6) -- (8);
\draw [-] (6) to [out=130,in=-130] (8);
\draw [-] (7) -- (8);
\draw [-] (8) -- (9);
\draw [-] (7) -- (9);
\draw [-] (9) -- (12);
\draw [-] (7) -- (10);
\draw [-] (7) to [out=40,in=140] (11);
\end{tikzpicture}
\end{center}
consists of 5 unbounded edges (halflines) and 13 bounded edges.
Therefore, according to Corollary~\ref{coroll}, for $\mu$ large enough
there at least 13 (positive) bound states of mass $\mu$.

On the other hand,  it follows from the results of \cite{ast}
(see in particular Theorem~2.5 and the remark after Theorem~2.3 therein)
that, regardless of $\mu$, this graph admits \emph{no ground state}:
as a consequence, none of these 13 bound states  is
a ground state.

More generally (replacing 13 with the actual number of bounded edges)
this applies to any graph $\G$  satisfying the so
called ``assumption (H)'' introduced in \cite{ast}, a topological assumption which
rules out ground states of any prescribed mass.
\end{example}

\begin{example}\label{example2} The following graph

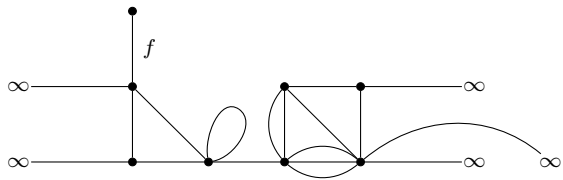
\begin{center}
\begin{tikzpicture}
\node at (0,0) [nodo] (1) {};
\node at (-1.5,0) [infinito]  (2){\infvert};
\node at (1,0) [nodo] (3) {};
\node at (0,1) [nodo] (4) {};
\node at (-1.5,1) [infinito] (5) {\infvert};
\node at (2,0) [nodo] (6) {};
\node at (3,0) [nodo] (7) {};
\node at (2,1) [nodo] (8) {};
\node at (3,1) [nodo] (9) {};
\node at (4.5,0) [infinito] (10) {\infvert};
\node at (5.5,0) [infinito] (11) {\infvert};
\node at (4.5,1) [infinito] (12) {\infvert};
\node at (0,2) [nodo] (13) {};
\draw [-] (4) -- node[right] {$\scriptstyle f$} (13) ;
\draw [-] (1) -- (2) ;
\draw [-] (1) -- (3);
\draw [-] (1) -- (4);
\draw [-] (3) -- (4);
\draw [-] (5) -- (4);
\draw [-] (3) -- (6);
\draw [-] (6) -- (7);
\draw [-] (6) to [out=-40,in=-140] (7);
\draw [-] (3) to [out=10,in=-35] (1.4,0.7);
\draw [-] (1.4,0.7) to [out=145,in=100] (3);
\draw [-] (6) to [out=40,in=140] (7);
\draw [-] (6) -- (8);
\draw [-] (6) to [out=130,in=-130] (8);
\draw [-] (7) -- (8);
\draw [-] (8) -- (9);
\draw [-] (7) -- (9);
\draw [-] (9) -- (12);
\draw [-] (7) -- (10);
\draw [-] (7) to [out=40,in=140] (11);
\end{tikzpicture}
\end{center}
is obtained from the graph of Example~\ref{example1} by adding one \emph{terminal} edge
(the one labelled ``$f$'' in the picture), so that
there are 14 bounded edges and therefore, according to
Corollary~\ref{coroll}, for $\mu$ large enough
$\G$ has at least 14 bound states of mass $\mu$.
Here, however, in contrast with Example~\ref{example1},
for large $\mu$ the presence of a terminal edge guarantees the existence
of a ground state of mass $\mu$
(see Proposition~4.1 of \cite{ast2}).

Now let us denote by $\{u_e\}$ the 14 bound states (one for every
bounded edge $e$ of $\G$) provided
by Corollary~\ref{coroll}, and recall that $u_e$ achieves its maximum on the edge $e$.

With this notation we claim that, among the bound states $\{u_e\}$, \emph{at most one}
can be a ground state,
the unique candidate being $u_f$ (the one that achieves its maximum
on the terminal edge $f$).

Indeed,  when $\mu$ is large  we can apply
Lemma~\ref{levelest} with $e=f$: since $f$ is a terminal edge and $u_f$ achieves the infimum in
\eqref{quasimezzosol}, in particular we obtain that $E(u_f,\G)<-\theta_p \mu^{2\beta+1}$.

Now consider any bounded edge $e\not=f$, and pick a point $x_0\in e$ where $u_e$ achieves its maximum.
 Since $x_0\not\in f$, exploiting the topology of $\G$ (more precisely,
of $\G\setminus f$) and the
fact $u_e(x)\to 0$ as $x\to \infty$
along any halfline, the reader can easily check that
\eqref{hpN} is satisfied if $w=u_e$ and $N=2$.
Then \eqref{toobig} gives
$E(u_e,\G)\geq -\theta_p \mu^{2\beta+1}$, so that $E(u_e,\G)>E(u_f,\G)$ and hence
$u_e$ cannot be a ground state.
\end{example}

\begin{example}\label{example3}  Consider the graph

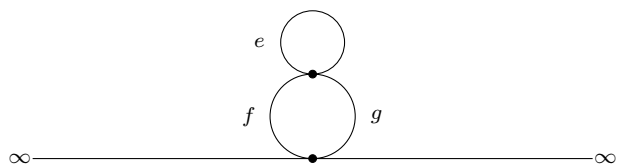
\begin{center}
\begin{tikzpicture}[xscale= 0.7,yscale=0.7]
\node at (-2,0) [infinito]  (0) {\infvert};
\node at (9,0) [infinito]  (100) {\infvert};
\node at (3.5,0) [nodo] (1) {};
\node at (3.5,1.6) [nodo] (2) {};
\draw [-] (0) -- (1);
\draw [-] (100) -- (1);
\draw(3.5,.8) circle (.8);
\draw(3.5,2.2) circle (.6);
\node at (2.3,.8) {$\scriptstyle f$};
\node at (4.7,.8) {$\scriptstyle g$};
\node at (2.5,2.2) {$\scriptstyle e$};
\end{tikzpicture}
\end{center}
with two halflines and three bounded edges $e$, $f$ and $g$ (note that $e$ forms a self-loop,
while $f$ and $g$ are assumed to have the same length).

For large $\mu$,
Corollary~\ref{coroll} yields three positive bound states $u_e$, $u_f$ and $u_g$,
each achieving its maximum on the corresponding edge. Of these bound states, however, only $u_e$
is a ground state: indeed, it is known (see Example~2.4 in \cite{ast}) that
for every mass $\mu$  this graph has a unique positive ground state, obtained by
fitting a soliton of mass $\mu$ on $\G$ with the maximum in the middle of the edge $e$
(thus respecting the symmetry of the graph). Therefore, in this graph,  the
minimization problem \eqref{minpV} (whose solution is of course $u_e$) is equivalent
to the minimization of the NLS energy with the mass constraint only, since the
additional requirement that $u$ achieves its maximum on $e$ is a \emph{natural constraint},
already satisfied by the ground state. More generally, these considerations
remain valid for all the graphs discussed
in Example~2.4 of \cite{ast}.
\end{example}

\begin{example}\label{example4} The graph

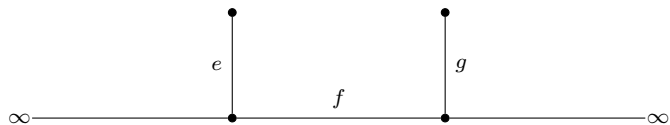
\begin{center}
\begin{tikzpicture}[xscale= 0.7,yscale=0.7]
\node at (-6,0) [infinito]  (0) {\infvert};
\node at (6,0) [infinito]  (100) {\infvert};
\node at (-2,0) [nodo] (1) {};
\node at (2,0) [nodo] (2) {};
\node at (-2,2) [nodo] (3) {};
\node at (2,2) [nodo] (4) {};
\draw [-] (0) -- (1) -- node[above] {$\scriptstyle f$} (2) -- (100);
\draw [-] (1) -- node[left] {$\scriptstyle e$}(3);
\draw [-] (2) -- node[right] {$\scriptstyle g$}(4);
\end{tikzpicture}
\end{center}
consists of two halflines and three bounded edges $e$, $f$ and $g$, hence for large $\mu$ one can find three bound states
$u_e$, $u_f$ and $u_g$,
each achieving its maximum on the corresponding edge.
We claim that, if $e$ and $g$ have the same length, then  both $u_e$ and $e_g$ are ground states,
while $u_f$ is not.

Indeed, since the two edges $e$ and $g$ are terminal,
for large $\mu$ there is at least one
ground state of mass $\mu$, with an energy level
\emph{lower} than the soliton threshold given by \eqref{ensol}
(see \cite{ast2}, Proposition~4.1).

On the other hand, since $u_f$ achieves its maximum on $f$ and $f$ is attached to a halfline at each
endpoint,  it is clear that \eqref{hpN} holds true with $w=u_f$ and $N=2$: then, we
see from \eqref{toobig} that
$E(u_f,\G)$ cannot be smaller than the energy threshold in \eqref{ensol}, so that $u_f$ cannot be
a ground state. Moreover, the same applies to any nonnegative function $w\in H^1_\mu(\G)$ that
achieves its maximum on a halfline of $\G$, so that any ground state must necessarily
achieve its maximum on $e$ or on $g$. Since $u_e$ (resp. $u_g$) has least energy among all functions
in $H^1_\mu(\G)$ with the maximum on $e$ (resp. on $g$), we see that one of them is a ground state:
finally, by the symmetry of the graph, $E(u_e,\G)=E(u_g,\G)$, so that both functions are a ground state.
\end{example}

\end{document}